\documentclass{amsart}
\usepackage[margin=1in]{geometry}

\usepackage{amsmath}
\usepackage{amsfonts}
\usepackage{amssymb,enumerate}
\usepackage{amsthm}
\usepackage{comment}
\usepackage{color}
\usepackage[all]{xy}
\usepackage{hyperref}
\usepackage{lineno}
\usepackage{graphicx}
\usepackage{setspace}
\usepackage{fancyhdr}
\usepackage{filecontents}

\newtheorem*{maintheorem*}{Main Theorem}
\newtheorem{theorem}{Theorem}[section]
\newtheorem{prop}[theorem]{Proposition}
\newtheorem{question}[theorem]{Question}

\theoremstyle{definition}
\newtheorem{definition}[theorem]{Definition}

\numberwithin{equation}{section}

\newcommand{\R}{\mathbb{R}}

\allowdisplaybreaks

\title{The Axial Electric Potential and Length of a Torus Knot}
\author{Henry Jiang}

\usepackage[backend=bibtex]{biblatex}
\begin{document}

\begin{abstract}
    Physical knot theory, where knots are treated like physical objects, is important to many fields. One natural problem is to give a knot a uniform charge, and analyze the resulting electric field and electric potential. There have been some results on the number of critical points of the electric potential from knots, such as by Lipton (2021) and Lipton, Townsend, and Strogatz (2022). However, little analysis has been done on the electric field and electric potential using calculations for specific knots.

    We focus on torus knots, specifically a parametrization that embeds it on a torus centered at the origin with rotational symmetry about the z-axis. Particularly, in this project, we analyze the electric field along the z-axis to take advantage of symmetry. We also analyze the length of the knot as a simpler integral. We show that the electric field is zero only at the origin, and investigate the extreme points of the electric field and electric potential using numerical methods and calculations. We also demonstrate a new way to apply methods for contour integration in complex analysis to calculate the length, electric potential, and electric field, and provide an explicit approximation for the length of a torus knot.
\end{abstract}

\maketitle

\section{Introduction and Goals}


In knot theory, a \emph{knot} is analogous to tying a piece of string in some way, then gluing the ends together. Classically, deforming a knot does not change the knot, and knots are considered to have no width. In physical knot theory, we cannot deform knots without changing its physical properties, so we must provide parametrizations of knots.

One intuitive problem that arises is adding a uniform charge on a knot. The resulting electric fields have been studied somewhat using general bounds (see \cite{lipt}) or simulations (see \cite{lipl}). However, one direction that has been studied very little is the analysis of electric fields from specific knots using rigorous calculations rather than approximations, which has only been done for the circle \cite{zyp}, but not for general knots. In addition, the the electric field and potential along the $z$-axis, the symmetric axis of the knot, have not been studied in depth for classes of knots like torus knots, especially by numerical methods. These are the main problems we will address. 

This problem and the results in our paper have many practical applications. The study of the electric field and electric potential of charged knots is relevant for making synthetic polymers in knotted configurations, which is used to make new materials tailored to specific uses (\cite{dom}). DNA, which is charged, also often takes the shape of a knot. Thus, analyzing properties of charged knots is relevant in studying enzymes that knot DNA (\cite{ars}). It is also relevant in gel electrophoresis of knotted DNA (which is used to differentiate strands of DNA) (\cite{weber}) as well as DNA-based computing, which uses DNA to store information. My results are therefore relevant to materials scientists, geneticists, and molecular biologists. 

In this paper, we have written a simulation to calculate the desired electric fields on the $z$-axis for various types of torus knots, and for various points on the $z$-axis. These numerical calculations have revealed patterns on the behavior of the electric field and potential. We have also proven some results using calculations, notably that the only zero of the electric field on the $z$-axis is at the origin. 

We have then applied some of the methods of complex analysis to the problem. Complex analysis methods have not been demonstrated on this problem in the past. Using our new method, we successfully found a good approximation for the length of a torus knot. We also show how this method could be applied to the electric potential and field. This provides a new method for calculating important properties of specific torus knots quickly and precisely. 

\section{Background on Torus Knots and Electric Fields}

Torus knots are the most intuitive knots to define, and are the most common types of knots that appear in practical applications. In physical knot theory, we must introduce a specific function that defines the torus knot we want to study. In this section, we introduce the torus knot and define properties such as the electric field and electric potential.

\subsection{Torus Knots}

In classical knot theory, we would not care about the exact shape of the knots as long as they are isomorphic: we consider knots to be a closed piece of string that can be moved freely without cutting it or passing it through itself. However, in physical knot theory, we must consider the actual shape of the knot in space, so we need an explicit equation for the knot.

We will focus on the \emph{torus knots}, which are defined as follows:
\begin{definition}
    Let $p$ and $q$ be two relatively prime positive integers. Then, the $(p,q)-$\emph{torus knot} is created by winding a string $q$ times through the hole of the torus and $p$ times around the larger circle of the torus. 
    
    Explicitly, one form of the $(p,q)$-torus knot can be given by the parametrization
    \begin{align*}
        x &= (\cos(qt) + 2)\cos(pt) \\
        y &= (\cos(qt) + 2)\sin(pt) \\
        z &= -\sin(qt)
    \end{align*}
    where $t$ ranges from $0$ to $2\pi$.
\end{definition}

Note that a torus rotationally symmetric about the $z$-axis can be defined with the equation \[\left(\sqrt{x^{2} + y^{2}} - R\right)^{2} + z^{2} = r\] for positive constants $R,r$. Here, $r$ is the radius of ``small'' circles formed by vertical cross-sections, whereas $R$ is the radius of the circle on which the centers of the small circles lie.

The specific parametrization introduced above always gives a point lying on the torus with $R = 2$ and $r = 1$, since
\begin{align*}
    \left(\sqrt{x^{2} + y^{2}} - R\right)^{2} + z^{2} &= \left((\cos(qt) + 2) - 2\right)^{2} + z^{2} \\
    &= \cos^{2}(qt) + \sin^{2}(qt) \\
    &= 1.
\end{align*}

This parametrization thus gives a torus knot that actually lies on a torus, so we will use this parametrization. In Figure 1, we can see the result from this parametrization for the torus knot with $p = 2$ and $q = 3$.

\begin{figure}[ht]
\centering
\includegraphics[scale=0.45]{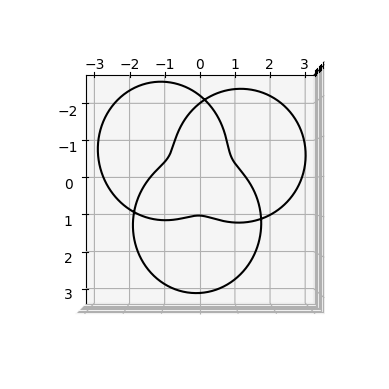}\hfill
\includegraphics[scale=0.35]{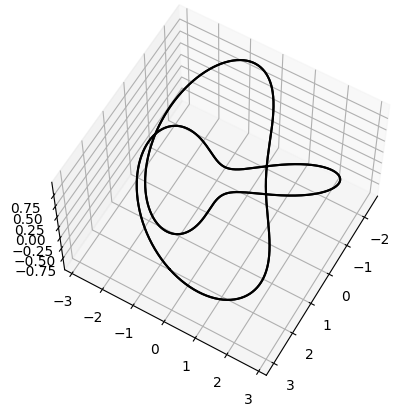}\hfill
\includegraphics[scale=0.45]{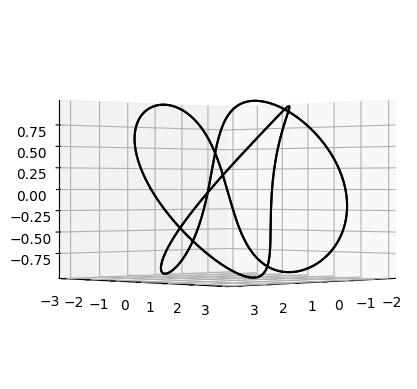}

\caption{A trefoil knot from the above parametrization at various angles}
\end{figure}

In Figure 2 is the $(3,8)$ torus knot, as a more complex example. Particularly, the parametrization gives very nice rotational symmetry which will be useful when calculating electric fields along the $z$-axis.
\begin{figure}[ht]
    \centering
    \includegraphics[scale=0.5]{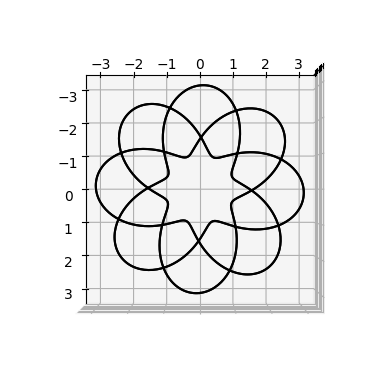}\hfill
    \includegraphics[scale=0.4]{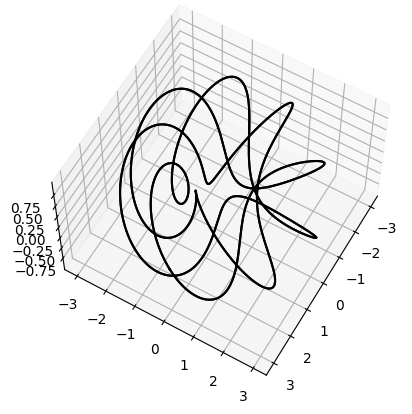}\hfill
    \includegraphics[scale=0.5]{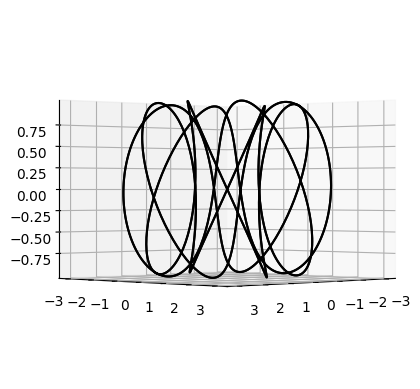}
    
    \caption{A (3,8)-torus knot from the above parametrization at various angles}
\end{figure}
 
\pagebreak

This parametrization is not symmetric in $p$ and $q$ which can create unexpected results in which, for example, a $(2,3)$ and $(3,2)$ torus knot look different. A $(3,2)$-torus knot, which is shown in figure 3, is isomorphic to the $(2,3)$ knot (they can be deformed to each other) but have very different physical properties.  In general, $(p,q)$ and $(q,p)$ are classically equivalent but are treated differently in our parametrization.

\begin{figure}[ht]
    \centering
    \includegraphics[scale=0.5]{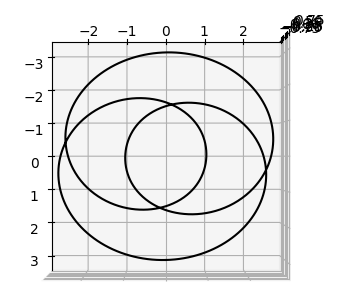}\hfill
    \includegraphics[scale=0.4]{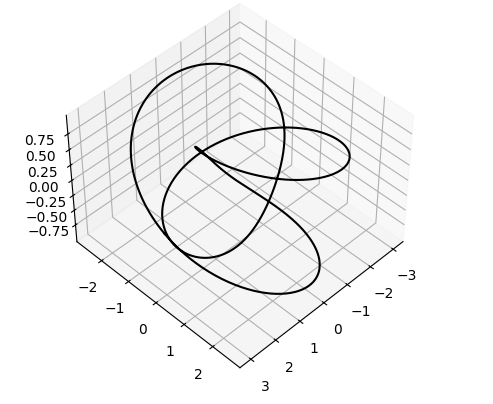}\hfill
    \includegraphics[scale=0.5]{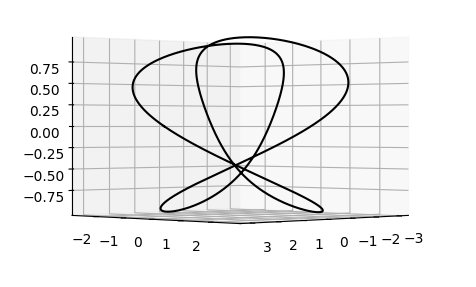}
    
    \caption{A (3,2)-torus knot from the above parametrization at various angles}
\end{figure}

Other specific parametrizations of the torus knots could be explored in the future, but this is the most intuitive one since it lies on a torus and is rotationally symmetric. 

\subsection{The Electric Potential}

The \emph{electric potential} to a point charge is proportional to $\frac{1}{r}$, where $r$ is the distance to the charge. Electric potential, denoted $\Phi$, is a scalar. Say that we have a curve $r(t)$, parametrized in terms of some variable $t$ which ranges from $0$ to $2\pi$. Since we are working with physical knots, we can assume the curve is in $\R^{3}$. Then, we may calculate the distance to some point $x$ in $\R^{3}$:
\begin{equation*}
    \Phi(x) = \int_{t = 0}^{2\pi} \frac{1}{|x - r(t)|}\ \mathrm{d}r = \int_{0}^{2\pi} \frac{|r'(t)|}{|x - r(t)|}\ \mathrm{d}t.
\end{equation*}
Note that the $|r'(t)|$ term is important, since the parametrizations above may not give constant $r'(t)$, so as $t$ increases, we may move along the knot at varying speeds. We include the term so that the integral is the same as long as the parametrization gives the same shape, since we are investigating the set defined by this curve and not the curve itself.

\subsection{The Electric Field}

The \emph{electric field} is defined as the gradient of the electric potential. We may differentiate inside the integral. If we let $x - r(t)$ have components $a_{1},a_{2},a_{3}$ then we obtain
\[\frac{\partial}{\partial a_{1}} \frac{1}{|x - r(t)|} = \frac{\partial}{\partial a_{1}} \frac{1}{\sqrt{a_{1}^{2} + a_{2}^{2} + a_{3}^{2}}} = -\frac{a_{1}}{(a_{1}^{2} + a_{2}^{2} + a_{3}^{2})^{3/2}}.\] We can find a similar expression for the $a_{2}$ and $a_{3}$ components of the gradient, which tells us
\begin{equation*}
    \nabla \Phi(x) = \int_{0}^{2\pi} \frac{(x - r(t))}{|x - r(t)|^{3}}|r'(t)| \ \mathrm{d}t.
\end{equation*}

A natural question to ask is as follows:
\begin{question}
    From a $(p,q)-$torus knot of uniform charge, where is the electric field zero?
\end{question}
Note that this is equivalent to finding $x$ such that the above integral is zero.


Another interesting question is:
\begin{question}
    How does the electric field at certain points change if $p$ or $q$ are increased?
\end{question}
This question does not require $p$ and $q$ to be relatively prime, as $p$ and $q$ sharing a factor only means that the resulting shape may not be one connected knot. Asking about its electric field is still valid.

To take advantage of symmetry, it is interesting to consider to only consider points on the $z$-axis. This gives several questions containing the behavior of the electric field in terms of $z$, and comparisons of the electric field between different $p$ and $q$.

Work has been done on this problem (\cite{lipl}) through simulations and theoretical bounds, but we will offer some new observations and methods.

\section{Branch cuts in Complex Analysis}

Complex analysis is a field that studies functions in the complex plane. The integrals from the previous section can be embedded in the complex plane, so methods in complex analysis are useful for evaluating them. In this section, we state some well-known theorems in complex analysis on \emph{analytic} functions that reduce complex integrals to simpler calculations.

We shall assume that the reader is familiar with basic terms in complex analysis.

In our particular problem, square roots are common. However, in the complex plane, there are two possible values for square roots, and therefore we must make more careful considerations.

\begin{definition}
    In the following sections, unless specified otherwise, the square root of a complex number will be the square root with argument in $[0,\pi]$. That is, the square root of $re^{i\theta}$ where $r$ is nonnegative and $\theta \in [0,2\pi)$ is defined as $\sqrt{r}e^{\frac{i\theta}{2}}$. 
\end{definition}
We must make note throughout all our calculations that no sign errors are introduced.

The following notions help us deal with functions with multiple values. Note that these definitions are more specific than the conventional definitions because our functions only have two possible values.
\begin{definition}[\cite{branch}]
    Consider a function $f(z)$ on the complex plane. A \emph{branch point} is a point that cannot be encircled by a curve $C$ such that $f$ is continuous and single-valued along $C$.
\end{definition}

\begin{definition}[\cite{branch}]
    A complex function $f(z)$ has a branch point at $\infty$ if $f(\frac{1}{z})$ has a branch point at $0$.
\end{definition}

Intuitively, a branch point causes the ``ambiguous'' part of the function (logarithms, roots, etc.) to go to $0$ or $\infty$. We also state the following convention:

For example, in the function $f(z) = \sqrt{z}$ on the complex plane, the two branch points are $0$ and the point at infinity.

If we want to define contour integrals in the plane, we must introduce \emph{branch cuts}.
\begin{definition}[\cite{branch}]
    A complete set of \emph{branch cuts} is a set of cuts in the plane with endpoints on branch points such that each branch point is an endpoint of a cut and every curve in the plane encloses all or none of the branch points.
\end{definition}

Contours must then go around branch cuts. A possible contour for the function $f(z) = \sqrt{z}$ is shown below. In this case, the branch cut is a segment from $0$ to $\infty$. Thus, if we wanted to integrate the function over the unit circle, a possible contour is the following:
\begin{figure}[ht]
    \centering
    \includegraphics[scale=0.6]{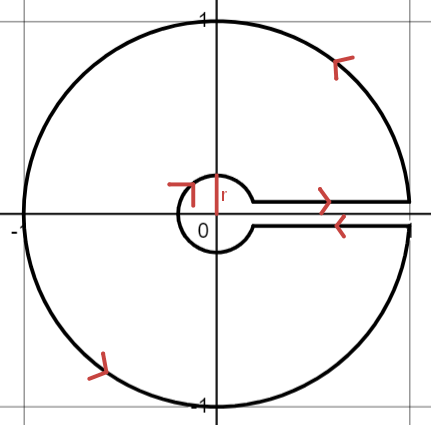}
    \caption{A possible contour for $\sqrt z$}
\end{figure}

\section{Numerical Analysis of Fields along the z-axis}

To use symmetry and get a better understanding of the electric field from knots, we may consider the electric field along the z-axis. In this section, we analyze the electric potential and electric field from the knot along the z-axis, as well as the length of the knot. Here, we focus on analysis using numerical methods and approximations, as well as some elementary calculations.

\subsection{Fields in the x and y directions}

We first deal with the $x$ and $y$ directions. We can deal with both cases simultaneously by considering the integral
\[\int_{0}^{2\pi} \frac{re^{pi\phi}(q^{2} + p^{2}r^{2})^{1/2}}{(\alpha^{2} + 2\alpha \sin(q\phi) + 5 + 4\cos(q\phi))^{3/2}} \ \mathrm{d}\phi,\] where the real and imaginary parts become the $x$ and $y$ components, repectively. To evaluate the integral, we realize that everything except the $e^{pi\phi}$ factor is periodic every $\frac{2\pi}{q}$. Thus, the integral is equal to
\[\left(1 + e^{\frac{2i\pi p}{q}} + e^{\frac{4i\pi p}{q}} + \cdots + e^{\frac{2(q - 1)i\pi p}{q}}\right)\int_{0}^{\frac{2\pi}{q}}  \frac{r(q^{2} + p^{2}r^{2})^{1/2}}{(\alpha^{2} + 2\alpha \sin(q\phi) + 5 + 4\cos(q\phi))^{3/2}}e^{pi\phi} \ \mathrm{d}\phi.\]
Noting now that $\gcd(q,p)$ must be $1$, the factor on the left just sums over all $q$th roots of unity, which gives $0$. Thus, the integral evaluates to $0$ so there is no component of the electric field in the $x$ or $y$ direction.

This gives us our first main result:
\begin{prop}
    The electric field from a point along the $z$-axis is parallel to the $z$-axis.
\end{prop}

\subsection{Field in the z-direction}

The most interesting component of the electric field is thus in the $z$-direction. 

If we set $x = y = 0$, we consider the electric field at at the point $(0,0,\alpha)$. It is equal to the following integral:
\[\int_{0}^{2\pi} \frac{(0,0,\alpha) - (r_{x},r_{y},r_{z})}{(r_{x}^{2}+r_{y}^{2}+(\alpha - r_{z})^{2})^{3/2}}|r'(t)| \ \mathrm{d}t.\]

This simplifies to
\[\int_{0}^{2\pi} \frac{(0,0,\alpha) - (r_{x},r_{y},r_{z})}{(\alpha^{2} + 2\alpha\sin(q\phi) + 5 + 4\cos(q\phi))^{3/2}}(q^{2} + p^{2}r^{2})^{1/2} \ \mathrm{d}\phi.\] 

If we substitute $u = q\phi$, the integral becomes
\[\int_{0}^{2\pi q} \frac{(\alpha - \sin(u))(q^{2} + p^{2}(2 + \cos(u))^{2})^{1/2}}{q(\alpha^{2} + 2\alpha\sin(u) + 5 + 4\cos(u))^{3/2}}\ \mathrm{d}u.\] Since the integral is periodic every $2\pi$, we can just consider it from $0$ to $2\pi$. Now, we may substitute $\omega = e^{iu}$, so $d\omega = ie^{iu}du$. Thus, we want
\[C\int_{|\omega| = 1} \frac{\left(\alpha - \frac{\omega - \frac{1}{\omega}}{2i}\right)\left(q^{2} + 4p^{2} + 2p^{2}\left(\omega + \frac{1}{\omega}\right) + \frac{p^{2}}{2}\left(\omega + \frac{1}{\omega}\right)^{2}\right)^{1/2}}{i\omega\left(\alpha^{2} + \alpha\left(\frac{i}{\omega} - i\omega\right) + 5 + 2\omega + \frac{2}{\omega}\right)^{3/2}} \ \mathrm{d}\omega = 0.\]

We may multiply top and bottom by $\omega^{3/2}$. Since both items in the square roots are currently real, there are no sign issues with multiplying both expressions by $\omega$. The demoniator is now
\[\omega(\omega^{2}(2 - \alpha i) + \omega(5 + \alpha^{2}) + (2 + \alpha i))^{3/2}.\] The quadratic term has roots $\omega = -\alpha i - 2$ and $\omega = \frac{1}{\alpha i - 2} = -\frac{\alpha i + 2}{\alpha ^{2} + 4}$ which can be found by the quadratic formula. Clearly (as $\alpha$ is real) only the latter is inside the $|\omega| = 1$ circle. 

We thus want
\[\int_{|\omega| = 1} \frac{(\omega^{2} - 2i\alpha \omega - 1)\sqrt{\left(q^{2} + p^{2}\left(\frac{\omega^{2}}{2} + 2\omega + \frac{1}{2}\right)^{2}\right)}}{2\omega^{3/2}\left((2-\alpha i)\left(\omega + 2 + \alpha i\right)\left(\omega + \frac{\alpha i + 2}{\alpha^{2} + 4}\right)\right)^{3/2}} \ \mathrm{d}\omega.\]

Finding this integral is still difficult, but progress can be made. The branch points are known: the denominator gives $-\frac{\alpha i + 2}{\alpha^{2} + 4}$ and $-2 - \alpha i$ as branch points, the former one being in the unit circle. The numerator gives $\omega + \frac{1}{\omega} = \pm 2\frac{q}{p}i - 4$.

We observe that $p$ and $q$ only appear in one term in the numerator. Particularly, we may notice that the only term containing $p,q$ is the $\sqrt{\frac{p^{2}}{2}\omega^{2} + \omega(q^{2}+2p^{2}) + \frac{p^{2}}{2}}$. If both $p$ and $q$ are scaled by some constant $c$, the entire integral is thus also scaled by $c$. Thus, only the ratio of $p$ and $q$ matters. It makes sense now to drop the condition that $p$ and $q$ are integers and vary $\frac{p}{q}$ freely.






\subsection{Critical points of the electric potential along the z-axis}

We show with an elementary calculation that the electric potential is maximized only at zero, and thus the electric field zero only at zero. This is another one of our main results.
\begin{theorem}
    The electric field along the $z$-axis from a torus knot is zero only at the origin.
\end{theorem}
\begin{proof}
    At the point $(0,0,\alpha)$, the electric potential from a $(p,q)$-torus knot is
    \[\int_{0}^{2\pi} \sqrt{\frac{q^{2} + p^{2}(2 + \cos(qt))^{2}}{(\cos(qt) + 2)^{2}(\cos^{2}(pt) + \sin^{2}(pt)) + (\alpha + \sin(qt))^{2}}} \ \mathrm{d}t = \int_{0}^{2\pi} \sqrt{\frac{q^{2} + p^{2}(2 + \cos u)^{2}}{\alpha^{2} + 2\alpha \sin(u) + 5 + 4\cos(u)}} \ \mathrm{d}u.\] Here, the $\frac{1}{q}$ of the $u$-substitution cancels with the changed integrals of the limit from $0$ to $2\pi q$, since $q$ is an integer and the integrand is periodic with period $2\pi$.

    We claim the maximum of this integral occurs at $\alpha = 0$ only. We will consider $f(u)$ to be the integrand. Particularly, we now claim the maximum of $f(u) + f(-u)$ is at $0$ for any $u \in [0,2\pi]$. Note that $f$ is periodic with $f(u) = f(2\pi + u)$, so the desired integral is equal to $\int_{-\pi}^{\pi}f(u)\ \mathrm{d}u = \int_{0}^{\pi}f(u) + f(-u) \ \mathrm{d}u$.

    First, we notice that the numerator, $\sqrt{q^{2} + p^{2}(2 + \cos u)^{2}}$ is the same as $\sqrt{q^{2} + p^{2}(2 + \cos(-u))^{2}}$ and is always positive. Thus, to find the maximum value of $f(u) + f(-u)$, it simply suffices to find the maximum of
    \[\frac{1}{\sqrt{\alpha^{2} + 2\alpha \sin(u) + 5 + 4\cos(u)}} + \frac{1}{\sqrt{\alpha^{2} - 2\alpha \sin(u) + 5 + 4\cos(u)}}.\] Let $m = 2\sin(u)$ and $n = 5 + 4\cos(u)$. 

    Then, we want to find the maxima of
    \[g(\alpha) := \frac{1}{\sqrt{\alpha^{2} + m\alpha + n}} + \frac{1}{\sqrt{\alpha^{2} - m\alpha + n}}.\] We want to set the derivative of $g$ to $0$, and we find
    \[g'(\alpha) = \frac{2\alpha + m}{(\alpha^{2} + m\alpha + n)^{3/2}} + \frac{2\alpha - m}{(\alpha^{2} - m\alpha + n)^{3/2}}\] Setting this function to zero, we want
    \[(\alpha^{2} - m\alpha + n)^{3}(2\alpha + m)^{2} = (\alpha^{2} + m\alpha + n)^{3}(2\alpha - m)^{2},\] which is obtained by squaring both sides and gathering terms. Notice that, since we squared both sides, some values of $\alpha$ satisfying this equation will not satisfy the original equation. We must be careful to deal with these cases in the end. 

    Notice that the two sides differ only in the sign of terms involving $m$. Thus, all terms with an even power of $m$ will cancel. Thus, we merely need to find the terms with an odd power of $m$. 
    \begin{itemize}
        \item The only term with a $m^{5}$ is $(-m\alpha)^{3}(m)^{2} = -m^{5}\alpha^{3}$.
        \item In a term with a $m^{3}$, we could have all $m$s coming from the $(\alpha^{2} - m\alpha + n)^{3}$ factor, giving us the term $-(m\alpha)^{3}(2\alpha)^{2} = -4m^{3}\alpha^{5}$.
        \item In a term with $m^{3}$, we could have two $m$s coming from the $(\alpha^{2} - m\alpha + n)^{3}$ factor and one from the $(2\alpha + m)^{2}$ factor, giving us $6m^{3}\alpha^{2}(\alpha^{2} + n)(2\alpha)$. The factor of $6$ comes from choosing which factor to take each term from.
        \item In a term with $m^{3}$, we could have one $m$ coming from the $(\alpha^{2} - m\alpha + n)^{3}$ factor and two from the $(2\alpha + m)^{2}$ factor, giving us $-3m^{3}\alpha(\alpha^{4} + 2\alpha^{2} n + n^{2})$.
        \item In a term with only one $m$, if the $m$ comes from $(\alpha^{2} - m\alpha + n)$, we get a term of $-3m\alpha(4\alpha^{2})(\alpha^{4} + 2n\alpha^{2} + n^{2})$.
        \item In a term with only one $m$, if the $m$ comes from $2\alpha + m$, we get $2(2\alpha)(\alpha^{6} + 3\alpha^{4}n + 3\alpha^{2}n^{2} + 3n^{3})$.
    \end{itemize}
    Combining these terms gives
    \begin{align*}
        0 = &-m^{5}\alpha^{3} \\
        &+ m^{3}(-4\alpha^{5} + 6(2\alpha^{5} + 2\alpha^{3}n) - 3(\alpha^{5} + 2\alpha^{3}n + \alpha n^{2})) \\
        &+ m(-12\alpha^{3}(\alpha^{4} + 2n\alpha^{2} + n^{2}) + 4\alpha(\alpha^{6} + 3\alpha^{4}n + 3\alpha^{2}n^{2} + 3n^{3})) 
    \end{align*}
    In fact, subtracting the two sides of the equation gave an extra factor of $2$, which we have removed.

    This evaluates to
    \[-m^{5}\alpha^{3} + m^{3}\alpha^{5} + 6m^{3}n\alpha^{3} - 3m^{3}\alpha n^{2} - 8m\alpha^{7} - 12mn\alpha^{5} + 4mn^{3}\alpha = 0\]
    We can factor out $\alpha$ and $m$ to get
    \[-8\alpha^{6} + \alpha^{4}(5m^{2} - 12n) + \alpha^{2}(6m^{2}n - m^{4}) + 4n^{3} - 3m^{2}n^{2} = 0.\] Call this polynomial $P(\alpha)$. Then, the equation is satisfied for roots of $P$.

    We must, however, deal with extraneous solutions introduced by squaring the equation. Particularly, to satisfy the original equation, the two terms in $g'(\alpha)$ must have opposite sign. Thus, $2\alpha + m$ and $2\alpha - m$ must have opposite sign, since the $(\alpha^{2} - m\alpha + n)^{3/2}$ and $(\alpha^{2} + m\alpha + n)^{3/2}$ terms are always positive. This means that $|\alpha| < \frac{m}{2}$. Notice also that $\frac{m}{2} = \sin(u) \le 1$. It is also nonnegative, since $0 \le u \le \pi$ Define a new polynomial by $Q(\alpha^{2}) = P(\alpha)$, or
    \[Q(\alpha) := -8\alpha^{3} + \alpha^{2}(5m^{2} - 12n) + \alpha(6m^{2}n - m^{4}) + 4n^{3} - 3m^{2}n^{2}.\] The roots of $P$ are the square roots of roots of $Q$. If a root of $Q$ was negative, it would not correspond to a root of $P$. If a root of $Q$ was greater than $1$, its square roots have absolute value greater than $1$, and thus greater than $\frac{m}{2}$. If a root of $Q$ was in $\left(\frac{m}{2},1\right)$, then its square roots would still have absolute value greater than $\frac{m}{2}$. Thus, we simply need to check that $Q$ has no roots between $0$ and $\frac{m}{2}$, inclusive. We will do by looking at $Q(0)$ and $Q\left(\frac{m}{2}\right)$ and showing that $Q$ cannot cross the $x$-axis between those two points.

    In the following calculations, we will make use of the $AM-GM$ inequality. It states that if $a_{1},a_{2},\ldots,a_{k}$ is some finite sequence of positive real numbers, then
    \[\frac{a_{1} + a_{2} + \cdots + a_{k}}{k} \ge \sqrt[k]{a_{1}a_{2}\cdots a_{k}},\] that is, the arithmetic mean is greater than the geometric mean. Equality holds only when $a_{1} = a_{2} = \cdots = a_{k}$.

    We first find that $Q(0) = 4n^{3} - 3m^{2}n^{2}$. We claim that this is positive. It suffices to prove that $4n - 3m^{2}$ is positive. Particularly, this is equal to $4(5 + 4\cos(k) - 3\sin^{2}(k)) = 4(2 + 4\cos(k) + 3\cos^{2}(k))$. Note that the polynomial $3x^{2} + 4x + 2$ has minimum at $x = -\frac{2}{3}$, giving a value of $\frac{2}{3}$, which means that the original function has a minimum value of $\frac{8}{3}$, which is indeed positive.

    Now, we look at $Q\left(\frac{m}{2}\right)$. We get \[-m^{3} + \frac{5m^{4}}{4} - 3m^{2}n + 3m^{3}n - \frac{m^{5}}{2} + 4n^{3} - 3m^{2}n^{2}.\] We claim this is positive. It suffices to show that \[-m^{3} + \frac{5m^{4}}{4} - 3m^{2}n + 3m^{3}n - \frac{m^{5}}{2} + \frac{8}{3}n^{2}\] is positive. Treating this expression as a quadratic in $n$, we note that it is an upward-facing parabola. Since $n = 5 + 4\cos(k) \ge 1$, we just want to show that the quadratic has no roots $n$ that are greater than $1$. Plugging in $1$ for $n$ gives $2m^{3} + \frac{3m^{4}}{4} - 6m^{2} + 4$. Since $m = 2\sin(k)$ and $0 \le k \le \pi$, $m$ is positive. Thus, we can use the AM-GM inequality to show
    \[2m^{3} + \frac{3m^{4}}{4} + 4 = 2 + 2 + x^{3} + x^{3} + \frac{3x^{4}}{4} \ge 5\sqrt[5]{3}m^{2} > 6m^{2}.\] Thus, the quadratic is positive at $1$. Now, notice that the vertex of the quadratic is at $\frac{9m^{2} - 9m^{3}}{4}$. Since $0 < m < 2$, we find that this is less than $1$ by the AM-GM inequality, which gives $\frac{4}{9} + m^{3} = \frac{4}{9} + \frac{m^{3}}{2} + \frac{m^{3}}{2} \ge \sqrt[3]{3}m^{2}$. Thus $Q\left(\frac{m}{2}\right)$ is also positive.

    Now we look at $Q'(\alpha)$ from $0$ to $\frac{m}{2}$. We have $Q'(\alpha) = -24\alpha^{2} + 2\alpha(5m^{2} - 12n) + 6m^{2}n - m^{4}$. At $0$, this is $6m^{2}n - m^{4}$. We already showed above that $4n - 3m^{2}$ is positive, so $6n - m^{2}$ is also positive (as $n$ is positive), so the only way $Q$ could have a root between $0$ and $\frac{m}{2}$ is if $Q'$ went positive, negative, and positive again in that interval (it cannot turn more than twice as $Q'$ is a quadratic). However, $Q'\left(\frac{m}{2}\right) = -6m^{2} + 5m^{3} - 12mn + 6m^{2}n - m^{4}$. Since $0 \le m \le 2$, both $-6m^{2} + 5m^{3} - m^{4}$ and $-12m + 6m^{2}$ are nonpositive, so this is also nonpositive, as $n$ is always positive. Thus $Q$ has no root in that interval so we are done.
\end{proof}

\subsection{Numerical Calculations of Integrals}

We also wrote code to calculate the integrals of complex functions such as the electric potential, electric field, and length over contours in the complex plane. Our code solves complex integrals over the unit circle by summing integrands over small pieces of the contour (akin to Riemann integration). 

We also added the ability to find maxima of integrals. This was done by assuming there is one maximum on positive inputs, then starting from $0$ and adding $1$ until reaching a local maximum, then adding or subtracting $0.5$, then $\frac{1}{4}$, and so on. As the jumps become smaller, the value converges on a local maximum. 

The code is available at \url{https://github.com/HenrySTEM/Electric-Potential-Torus-Knots}.

\subsection{Numerical Observations}

From numerical calculations, we have found some observations and calculations which could be proven rigorously in future research.

First, we found the shape of the electric field graph along the $z$-axis. The graph in Figure 5 shows the electric field at various $z$-values for $p = 3$ and $q = 2$.
\begin{figure}[ht]
    \centering
    \includegraphics[scale=0.6]{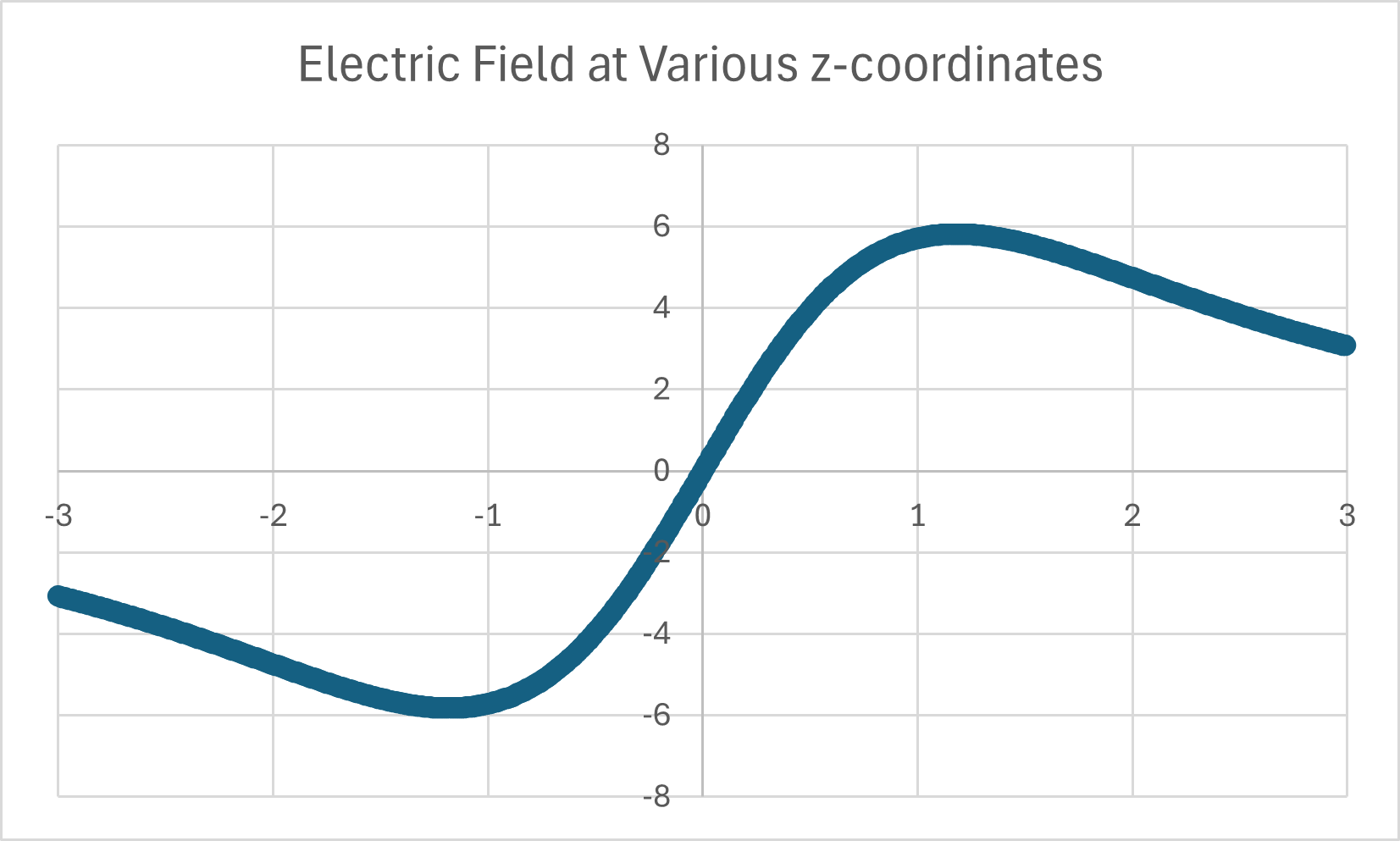}
    \caption{Electric field at $z$-coordinates from $-3$ to $3$ for the trefoil}
\end{figure}

The function is odd, which is not surprising because of the symmetries of the trefoil. However, an interesting question is
\begin{question}
    Is there only one maximum magnitude for the electric field (considering only positive z)?
\end{question}

Assuming this is true, we found locations of local maxima of the electric field. Recall that scaling $p$ and $q$ (together) do not change the location of the maximum. Since only the ratio of $p$ and $q$ determines what $z$-value gives the maximum electric field, can scale so that $p + q = 1$ and vary $p$ from $0$ to $1$ to make a plot. The results are shown in Figure 6.
\begin{figure}[ht]
    \centering
    \includegraphics[scale=0.8]{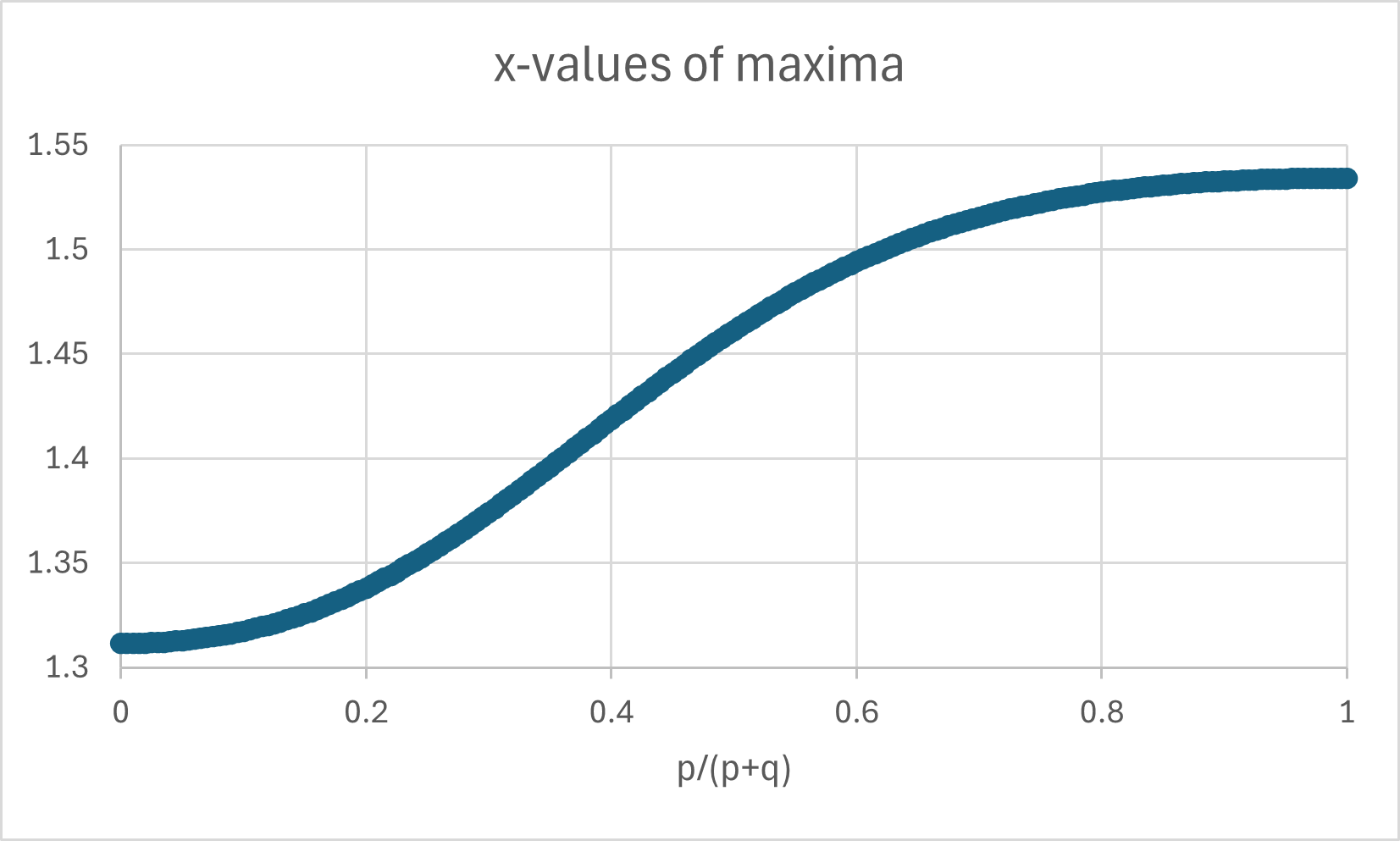}
    \caption{Location of maximum electric field for various $p$.}
\end{figure}

This gives a few questions:
\begin{question}
    Does the maximum electric field occur farther away when $p$ is increased with relation to $q$? The data suggest so.
\end{question}

\begin{question}
    What are the symmetries in the locations of the maxima?
\end{question}

Finally, we see if the strength of the electric field is stronger for certain torus knots compared to others. To check this, we divided by the length of the knot (also calculated by simulation) which gave a value for the electric fields at various points, such as $(0,0,1)$, which is shown in Figure 7. Again, only the ratio of $p$ to $q$ matters, so we assumed $p + q = 1$. 

\begin{figure}[ht]
    \centering
    \includegraphics[scale=0.8]{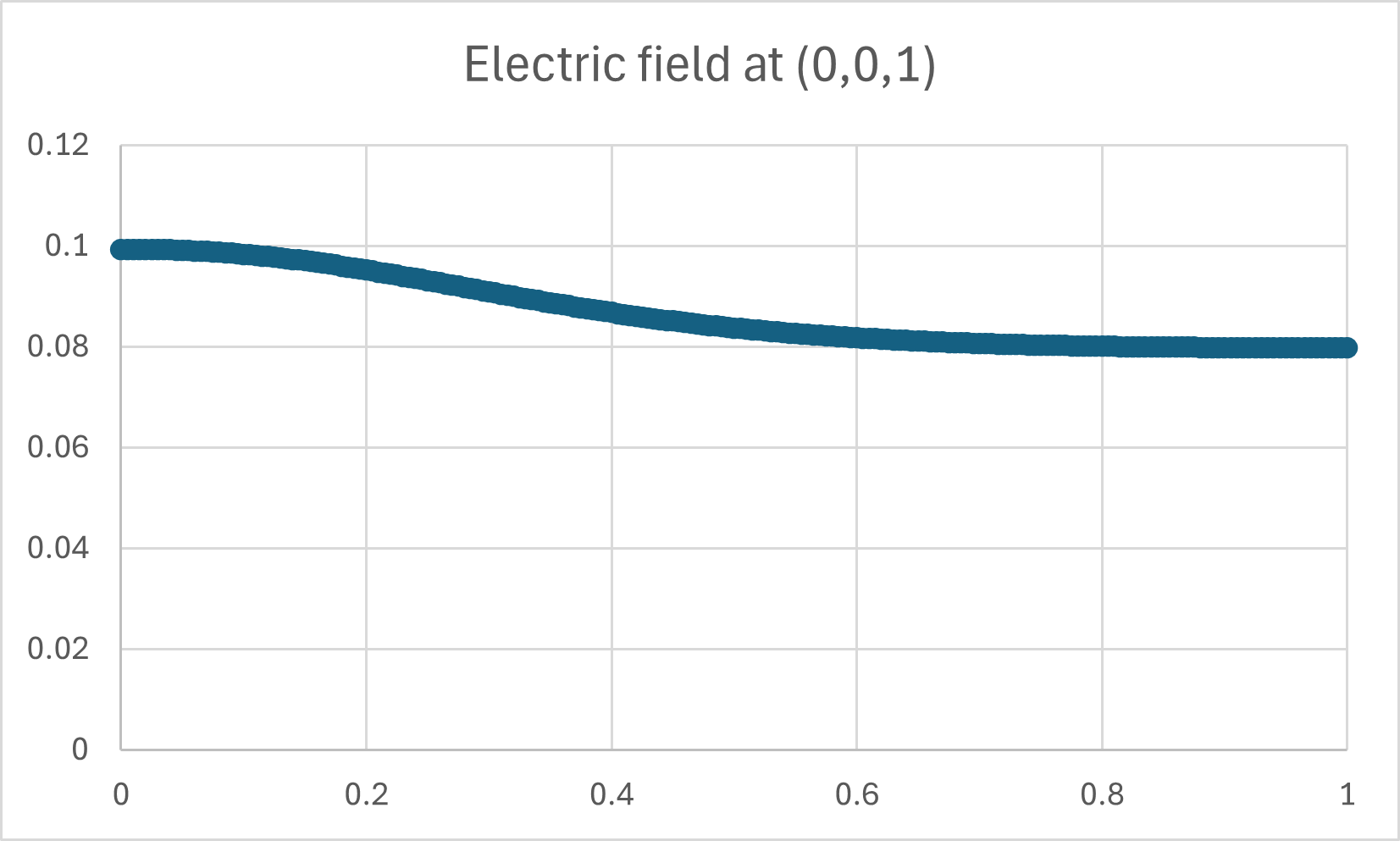}
    \caption{Electric field at $z$-coordinate $1$ for various $p$ (on the bottom)}
\end{figure}

Explicitly, this is the graph of
\[\frac{\int_{0}^{2\pi} \sqrt{\frac{q^{2} + p^{2}(2 + \cos u)^{2}}{\alpha^{2} + 2\alpha\sin(u) + 5 + 4\cos(u)}} \ \mathrm{d}u}{\int_{0}^{2\pi} \sqrt{q^{2} + p^{2}(2 + \cos u)^{2}} \ \mathrm{d}u}\] which is the electric potential divided by the length. 

\section{Analytic Calculations on the Length of the Torus Knot}

In this section, we show a demonstration of complex analysis methods by finding an explicit approximation for the length of the torus knot. 

This is a relevant calculation because we need it to study different knots that all have the same charge, since we must divide by the length of the knot to find the charge per unit length. Our methods using complex analysis provide a good closed-form approximation for the length of the torus knot, and we demonstrate that this is a viable method for calculating integrals related to the torus knot. Complex analysis methods have not previously been applied to this problem, and our results demonstrate this new approach.


\subsection{The Length Integral}

The length of the knot is equal to 
\[\int_{0}^{2\pi} |r'(t)| \ \mathrm{d}t = \int_{0}^{2\pi} \sqrt{q^{2} + p^{2}(2 + \cos(u))^{2}} \ \mathrm{d}u.\] Note that the integrand is periodic every $2\pi$, so this is equal to
\[p\int_{-\pi}^{\pi}\sqrt{\frac{q^{2}}{p^{2}} + (2 + \cos(u))^{2}} \ \mathrm{d}u.\] Now, let $k = \frac{q}{p}$. Setting $\omega = e^{iu}$ gives $\ \mathrm{d}\omega = i\omega \ \mathrm{d}u$ so our integral becomes
\[p\int_{|\omega| = 1}\frac{\sqrt{k^{2} + \left(2 + \frac{1}{2\omega} + \frac{\omega}{2}\right)^{2}}}{i\omega}\ \mathrm{d}\omega = \frac{p}{2}\int_{|\omega| = 1}\frac{\sqrt{4k^{2}\omega^{2} + \left(\omega^{2} + 4\omega + 1\right)^{2}}}{i\omega^{2}}\ \mathrm{d}\omega\] where the integral is over the unit circle, starting from $-1$ and going clockwise. 

\subsection{Finding Roots and Poles}

We notice that there is a pole at $\omega = 0$ of order $2$. This pole is not a branch point. There are four branch points at the roots of the numerator, which are at $\pm \sqrt{-k^{2} - 4ki + 3} \pm ik - 2$, where the two $\pm$ are independent. We note that two of these are in the unit circle and two are outside it. We also notice that this gives two sets of conjugates. 

We want to know where the poles are, so we may calculate them explicitly. The square root gives us 
\[\pm_{1} \left((k^{2} + 3)^{2} + (4k)^{2}\right)^{1/4}\exp\left(\frac{1}{2} i \left(\tan^{-1}\left(\frac{4k}{k^{2} - 3}\right) + K\right) \right)\pm_{2} ik - 2\] where $K = 0$ if $3 - k^{2}$ is positive (that is, $k < \sqrt 3$) and $K = \frac{\pi}{2}$ otherwise. We note that $\cos\left(\frac{\theta}{2}\right) = \sqrt{\frac{1 + \cos(\theta)}{2}}$ and $\sin\left(\frac{\theta}{2}\right) = \sqrt{\frac{1 - \cos(\theta)}{2}}$, and that $\cos^{2}(\theta) = \frac{1}{1 + \tan^{2}(\theta)}$. We can take the square root: the $K$ term ensures that $\cos\left(\frac{1}{2} i \left(\tan^{-1}\left(\frac{4k}{k^{2} - 3}\right) + K\right) \right)$ will be positive. 

By substituting this in, we then obtain the roots (inside the unit circle) at
\[\sqrt{\frac{\sqrt{(k^{2} + 1)(k^{2} + 9) + 3 - k^{2}}}{2}} - 2 \pm i\left(\sqrt{\frac{\sqrt{(k^{2} + 1)(k^{2} + 9) - 3 + k^{2}}}{2}} - k\right)\] with the real and imaginary parts swapped if $k > \sqrt3$. This shows that we can find the positions explicitly in terms of $k$. More importantly, we may bound the real part between $0$ and $-2 + \sqrt 3$ and the imaginary part between around $-0.15$ and $0.15$ The exact bounds are approximately on an ellipse, as shown in Figure 8 (the blue points plotted are the branch points for $k = 0,0.1,0.2,\ldots,1.2$)

\begin{figure}[h]
    \centering
    \includegraphics[scale=1]{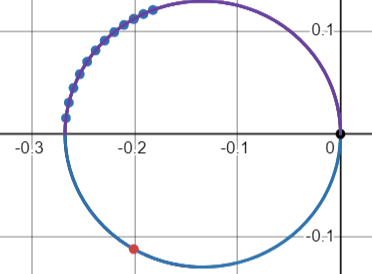}
    \caption{Locations of branch points for the length of the torus knot}
\end{figure}

Let the roots inside the unit circle be $m - ni$ and $m + ni$ where $n$ is positive. 

\subsection{Making Branch Cuts}

We make branch cuts vertically from each of these points, and both branch cuts go out from the negative real axis to connect to the other points. The two horizontal segments of the contour cancel since we have two branch cuts between them, and thus they are the same sign but in opposite directions. Thus, to determine the original integral, we need the residue at $0$ (so that we can use Cauchy's residue theorem) and the integral on the vertical segment. The resulting contour and branch cuts are shown in Figure 9.

\begin{figure}[h]
    \centering
    \includegraphics[scale=0.6]{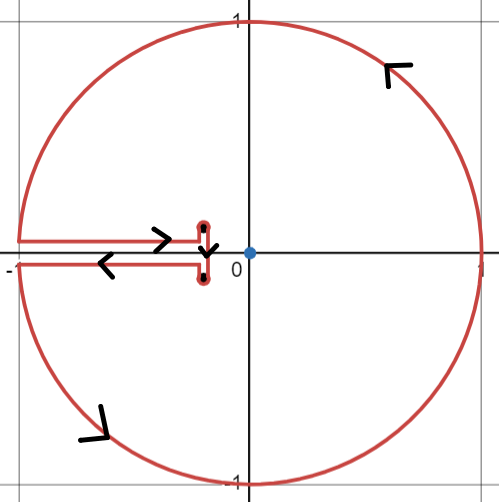}
    \caption{The contour for finding the length of a torus knot}
\end{figure}

From the formula for a residue, the residue is equal to the derivative of $\sqrt{4k^{2}\omega^{2} + (\omega^{2} + 4\omega + 1)^{2}}$ evaluated at $0$. The derivative is \[\frac{8k^{2}\omega + 2(\omega^{2} + 4\omega + 1)(2\omega + 4)}{2\sqrt{4k^{2}\omega^{2} + (\omega^{2} + 4\omega + 1)^{2}}}\] which is $4$ when $\omega = 0$. 

We can use the substitution $\omega = m + xi$ for the integral on the segment. Then $d\omega = idx$ so we simply get
\[\int_{|\omega| = 1}\frac{\sqrt{4k^{2}\omega^{2} + \left(\omega^{2} + 4\omega + 1\right)^{2}}}{i\omega^{2}}\ \mathrm{d}\omega = 8\pi + 2\int_{-n}^{n} \frac{\sqrt{q^{2} + p^{2}\left((m + xi)^{2} + 4(m + xi) + 1\right)^{2}}}{(m + xi)^{2}} \ \mathrm{d}x.\] The last term is multiplied by two because we traverse the segment between the two poles twice: once going up and another time going down, which do not cancel because of the branch cut. The $8\pi$ comes from $2\pi$ times $4$ by Cauchy's residue theorem.

\subsection{Approximating the Integral}

We have now reduced the problem to an integral on a segment. We may expand the square of the numerator as
\begin{align*}
    &\phantom{+}x^{4} \\
    &+ x^{3}(-4im - 8i) \\
    &+ x^{2}(-24m - 4k^{2} - 18) \\
    &+ x(4im^{3} + 24im^{2} + 8imk^{2} + 36im + 8i) \\
    &+ m^{4} + 8m^{3} + 4m^{2}k^{2} + 18m^{2} + 8m + 1.
\end{align*}
We have $-n \le x \le n$, and $n < 0.15$, and $-0.3 < m < 0$. The $x^{4}$ and $x^{3}$ terms are thus negligible (since the magnitude of the $x^{3}$ term is at least $-10$ and is not large enough to make it relevant compared to the $x^{2}$ term). We can also use the fact that $x = n$ is a root of this, so 
\[(4im^{3} + 24im^{2} + 8imk^{2} + 36im + 8i) = n^{2}(-4im - 8i).\]

Meanwhile, the constant term is equal to $-n^{4} - n^{2}(-24m - 4k^{2} - 18)$. Since $n^{4}$ is negligible, the constant term is around $n^{2}(-24m - 4k^{2} - 18) < -10n^{2}$. Thus the constant term is greater than $10n^{2}$ while the linear term is negative and greater than $-8n^{2}i$. Since $n$ is small, the linear term is thus negligible except near the roots (which we can ignore since the function is close to $0$ there). Thus, we look at just the constant and quadratic terms. This effectively sets the numerator equal to an ellipse. 

Thus, we can write the desired integral as
\[\int_{-n}^{n} \frac{\sqrt{b - ax^{2}}}{(m + ix)^{2}} \ \mathrm{d}x\] where $a,b,m,n$ are values that we can find explicitly in terms of $k$. This integral can be calculated analytically or approximated with a series. We do not show the details of the calculation here. However, one may find that this is equal to
\[\left.\frac{iam}{\sqrt{am^{2} + b}}\ln\left(\frac{m(\sqrt{b - ax^{2}} - \sqrt{b}) - x\sqrt{-am^{2} - b} - x\sqrt{b}i}{m(\sqrt{b - ax^{2}} - \sqrt{b}) + x\sqrt{-am^{2} - b} - x\sqrt{b}i}\right) + \sqrt{a}\arcsin\left(\frac{\sqrt a x}{\sqrt b}\right) + \frac{\sqrt{b - ax^{2}}}{x - im}\right|_{x = -n}^{n}\]

Thus, we used complex analysis to reduce our original integral to one which is much easier to calculate.

\section{Applying Complex Analysis to the Electric Potential}

Our complex analysis method is also applicable to the electric field and electric potential. In this section, we show that our method is applicable to the general problem. While will not perform the calculations here, we will demonstrate potential branch cuts and contours to use for electric field and potential calculations. 

\subsection{The Electric Potential Integral}

As we have seen before, the electric potential is
\[\int_{0}^{2\pi} \sqrt{\frac{q^{2} + p^{2}(2 + \cos u)^{2}}{\alpha^{2} + 2\alpha\sin(u) + 5 + 4\cos(u)}} \ \mathrm{d}u.\] If we substitute $\omega = e^{iu}$, we obtain $d\omega = i\omega du$ so the electric potential is equal to
\[\int_{|\omega| = 1} \sqrt{\frac{q^{2} + p^{2}(2 + \frac{\omega}{2} + \frac{1}{2\omega})^{2}}{i\omega \left(\alpha^{2} + 2\alpha\left(\frac{\omega}{2i} - \frac{1}{2i\omega}\right)+ 5 + 2\left(\frac{1}{\omega} + \omega\right)\right)}} \ \mathrm{d}u.\]
Where $|\omega| = 1$ refers to any loop around the unit circle - since the function of $u$ is periodic every $2\pi$, it does not matter where the loop starts. This is equal to
\[\int_{|\omega| = 1} \frac{\sqrt{4q^{2} + p^{2}\left(\omega^{2} + 4\omega + 1\right)^{2}}}{i\omega^{3/2}\left((\omega^{2}(2 - \alpha i) + \omega(5 + \alpha^{2}) + (2 + \alpha i)\right)^{1/2}} \ \mathrm{d}\omega.\]

As we saw in section 4.2, the denominator can be factored, giving us poles (and branch points) at $\omega = -2 - \alpha i$ and $\omega = -\frac{\alpha i + 2}{\alpha^{2} + 4}$.

The numerator gives the same branch points as we saw in the length of the torus knot. In addition, $0$ and $\infty$ are now branch points because the exponent of $\omega$ is $\frac{3}{2}$. One can see that the electric field has the same branch points.

Of these branch points, four are inside the unit circle. One of them is at the origin. The point $-\frac{\alpha i + 2}{\alpha^{2} + 4}$ lies on the circle $\left(\frac{1}{4} + x\right)^{2} + y^{2} = \frac{1}{16}$. The remaining two are the same as in the length function as we saw in section 5.1. Figure 10 shows all these points.

\begin{figure}[ht]
    \centering
    \includegraphics[scale=1.1]{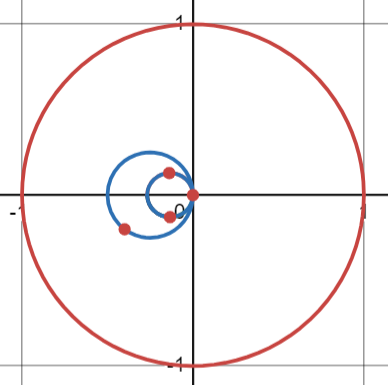}
    \caption{Branch points of the electric potential and curves they lie on}
\end{figure}

From connecting these points, we get a possible branch cut between these points. One of the promising branch cuts is shown in Figure 11.
\begin{figure}[ht]
    \centering
    \includegraphics[scale=0.8]{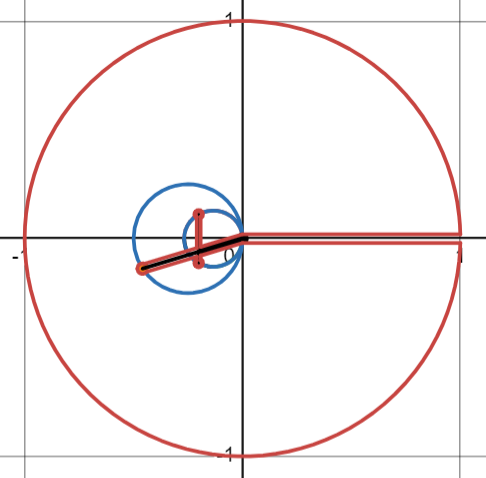} \hfill \includegraphics[scale=0.5]{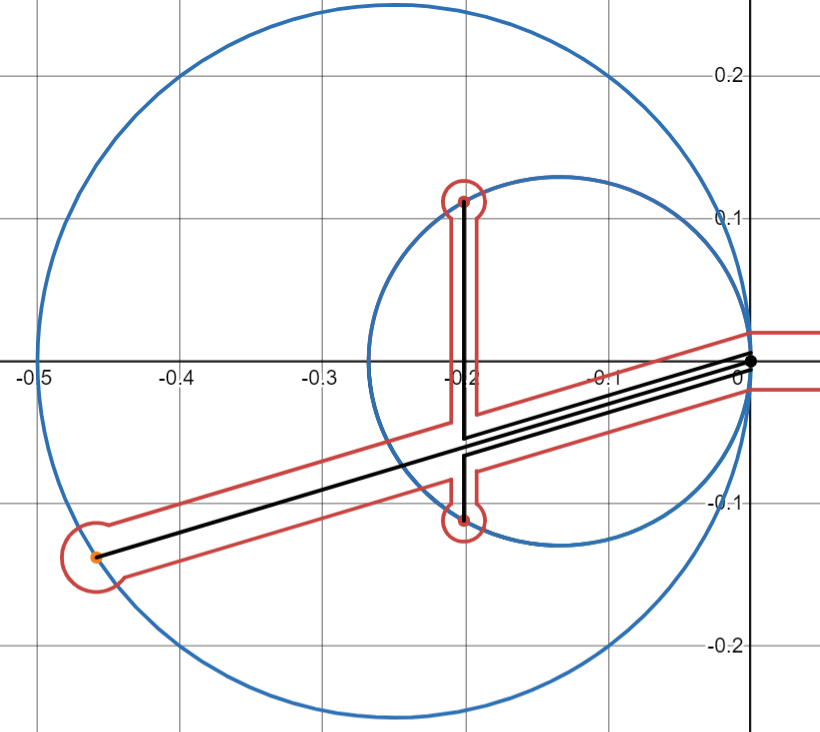}
    \caption{A possible branch cut for the electric potential}
\end{figure}

\section{Conclusion and Next Steps}

In our research, we have found trends in the properties of torus knots through numerical methods. We have discovered that the electric potential (per unit charge on the knot) and the location of the maximum electric field on the symmetric axis of the knot ($z$-axis) both increase as $p$ increases - that is, the knot winds more around the vertical circle (through the hole of the torus). We have also discovered and proven that the electric field along the $z$-axis is zero only at the origin.

In addition, we have introduced a new method for calculating these properties explicitly using contour integration and complex analysis. These methods simplify calculations in physical knot theory. Particularly, reduced the length of a torus knot to an integral along a segment, and found a good approximation using that integral. We also demonstrated how this method could be applied to the electric field and potential. 

One possible further direction for research is to consider integrals off the $z$-axis, so that all the zeroes can be located with numerical methods. Another direction is to consider other types of knots. Our complex analysis methods can simplify calculations given approximations for branch points, and are promising for future research.

\section*{Acknowledgements}

I would like to thank my mentor, Max Lipton, for his continued advice, knowledge, and guidance throughout this project. I would also like to thank the MIT PRIMES-USA program and organizers for giving me this great research opportunity.

\printbibliography

\end{document}